\documentclass{amsart}
\usepackage{amssymb,bbm,color,xcolor}
\usepackage{amscd}
\usepackage{enumerate}
\usepackage{siunitx}
\usepackage{bm}
\usepackage{ulem}
\usepackage{float, graphicx}
\usepackage{subfigure}
\numberwithin{equation}{section}

\usepackage[pagebackref]{hyperref}

\usepackage{mathtools}
\usepackage[tableposition=top]{caption}
\usepackage{booktabs,dcolumn}

\DeclareFontFamily{OT1}{rsfs}{}
\DeclareFontShape{OT1}{rsfs}{n}{it}{<-> rsfs10}{}
\DeclareMathAlphabet{\mathscr}{OT1}{rsfs}{n}{it}

\theoremstyle{plain}

\newtheorem{theorem}{Theorem}[section]

\newtheorem{lem}[theorem]{Lemma}

\theoremstyle{definition}

\newcommand{\bal}{\[\begin{aligned}}
\newcommand{\eal}{\end{aligned}\]}
\newcommand{\beeq}{\begin{equation}}\newcommand{\eneq}{\end{equation}}
\newcommand{\beqa}{\begin{eqnarray*}}\newcommand{\eeqa}{\end{eqnarray*}}
\newcommand{\beeqa}{\begin{eqnarray}}\newcommand{\eneqa}{\end{eqnarray}}

\def\<{\langle}             \def\>{\rangle}
\newcommand{\al}{\alpha}    \newcommand{\be}{\beta}
  \newcommand{\vep}{\varepsilon}
  \newcommand{\La}{\Lambda}
    \newcommand{\la}{\lambda}
\newcommand{\sig}{\sigma}  
  \newcommand{\Si}{\Sigma}
\newcommand{\om}{\omega}    
\newcommand{\gam}{\gamma}
\newcommand{\R}{\mathbb{R}}

\newcommand{\Sp}{\mathbb{S}}
\newcommand{\ms}{\mathbb{S}}
\newcommand{\pt}{\partial_t}
\newcommand{\les}{{\lesssim}}

\newcommand{\hn}{\mathbb{H}^n}
\newcommand{\hth}{\mathbb{H}^3}

\newcommand{\De}{\Delta}
\newcommand{\de}{\delta}

\newcommand{\hf}{\frac{1}{2}}

\newcommand{\gm}{\Gamma^{-}(O,t_0)}
\newcommand{\gp}{\Gamma^{+}(O,t_0)}

\newcommand{\one}{\uppercase\expandafter{\romannumeral1}}
\newcommand{\two}{\uppercase\expandafter{\romannumeral2}}
\newcommand{\three}{\uppercase\expandafter{\romannumeral3}}
\newcommand{\four}{\uppercase\expandafter{\romannumeral4}}
\newcommand{\five}{\uppercase\expandafter{\romannumeral5}}
\newcommand{\six}{\uppercase\expandafter{\romannumeral6}}
\newcommand{\sev}{\uppercase\expandafter{\romannumeral7}}
\newcommand{\eig}{\uppercase\expandafter{\romannumeral8}}
\newcommand{\wt}{\widetilde}
\newcommand{\ol}{\overline}

\title
[Wave equations with logarithmic nonlinearity]{
Wave equations with logarithmic nonlinearity on hyperbolic spaces
}

\author{Chengbo Wang}
\address{School of Mathematical Sciences\\ Zhejiang University\\Hangzhou 310058, P. R. China}\email{wangcbo@zju.edu.cn }

\author{Xiaoran Zhang$^{*}$}\thanks{* Corresponding author}
\address{School of Mathematical Sciences\\ Zhejiang University\\ Hangzhou 310058,P.R.China}
\email{1025391337@qq.com}

\thanks{The authors were supported by
 NSFC 11971428  and  NSFC 12141102. }

\date{\today}

\begin{document}

\bibliographystyle{plain}

\begin{abstract}
In light of the exponential decay of solutions of linear wave equations on hyperbolic spaces $\hn$, to illustrate the critical nature, we investigate nonlinear wave equations with logarithmic nonlinearity, which behaves like
$\left(\ln {1}/{|u|}\right)^{1-p}|u|$ near $u=0$, 
on hyperbolic spaces.
Concerning the global existence vs blow up with small data, we expect that the problem admits a critical power $p_c(n)>1$.
When $n=3$,
we prove that the critical power is $3$, by proving global existence for $p>3$, as well as generically blow up for $p\in (1,3)$.
\end{abstract}



\maketitle

\section{Introduction}

Let $n\ge 2$, consider the wave equation
\beeq\label{nlw}
\begin{cases}
\partial_{t}^2 u-(\De_{\hn}+\rho^2)u=F(u)\ ,\\
u(0,x)=\vep u_0, u_t(0,x)=\vep u_1\ ,
\end{cases}
\eneq
where $\rho=\frac{n-1}{2}$ (recall that the spectrum of $-\De_{\hn}$ is $[\rho^2,\infty)$), $u_0, u_1$ are smooth functions with compact support. As is well known, the global existence vs blow-up for nonlinear wave equations with power-type nonlinearities $F(u)\sim |u|^p$ is related to the so-called $Strauss\ conjecture$ in $\R^n$, which has a critical power $p_c(n)>1$. Correspondingly on hyperbolic spaces  $\hn$, it is known to admit global solutions for sufficiently small $\vep>0$, for any power $p\in (1, 1+4/(n-1))$, thanks to the improved 
 decay of solutions of linear wave equations.  In some sense, in handling the power nonlinearity, we do not need to explore the precise information on the decay rate and no critical phenomenon appears.

To capture the critical nature, in this paper, we propose 
the investigation of nonlinear wave equations with logarithmic nonlinearities $F_p(u)$ near $u=0$, for which we expect to have a critical power $p_c(n)>1$.

The interest arises from the similar equation in Euclidean spaces
\beeq\label{esnlw}
\begin{cases}
\partial_{t}^2 u-\De_{\R^n}u= |u|^p\ ,\\
u(0,x)=\vep u_0,\ u_t(0,x)=\vep u_1.
\end{cases}
\eneq
It has been studied for a long time and admits a critical power $p=p_c(n)>1$, such that for any compactly supported initial data
with sufficiently small size ($\vep\ll 1$), a regular global solution exists when $p>p_c(n)$, while such a result fails when $1<p<p_c(n)$. The first work in this direction is \cite{MR535704} in 1979 when $n=3$, where F. John determined the critical power $p_c(3)=1+\sqrt{2}$. Then Strauss \cite{MR614228} conjectured that the critical power $p_c(n)$ for other dimensions $n\ge2$ should be the positive root of the quadratic equation
$$(n-1)p^2-(n+1)p-2=0.$$
The conjecture was verified in Glassey \cite{MR631631}, \cite{MR618199} when $n=2$ with $p_c(2)=(3+\sqrt{17})/2$. Then for other dimensions, the existence portion of the conjecture was proved by Zhou \cite{MR1331521} ($n=4$), Lindblad-Sogge \cite{MR1408499} ($n\le8$) and Georgiev-Lindblad-Sogge \cite{MR1481816}, Tataru \cite{MR1804518} (all $n$, $p_c(n)<p\le p_{conf}$), where
$$p_{conf}(n)=1+\frac{4}{n-1}$$
is the conformal power. While the blow-up portion is due to Sideris \cite{MR744303} ($n\ge4$, $1<p<p_c(n)$).

On hyperbolic spaces $\hn$, in
geodesic polar coordinates, the metric is given by $g_{\hn}=dr^2+(\sinh r)^2d\om^2$, for $(r,\omega)\in (0,\infty)\times\Sp^{n-1}$,
 see Section \ref{pre}. If we take the same power-type nonlinearities $F(u)= |u|^p$, heuristically, we expect that small data global existences always hold for all $p>1$, 
 due to the $\sinh r$ factor in the metric. It is first proved by Fontaine \cite{MR1443052} in 1997 when $n=2,3$ from the perspective of Lie algebra, for any data $u_0\in C^1(\hn),u_1\in C(\hn)$ satisfying
\beeqa\label{data}
|u_0|+|u_1|+|\nabla u_0|\le  \theta_k\ , 
\eneqa
where $\nabla f=(\partial_r f\partial_r,(\frac{1}{\sinh r})^2\partial^{\om}f\partial_{\om})$, 
$|\nabla f|=((\partial_r f)^2+|\frac{\partial_{\om}f}{\sinh r}|^2)^{1/2}$ and $\theta_k =(\cosh |\cdot|)^{-k-\rho}$, $k>0$. 


For general spatial dimensions $n\ge 2$, 
Anker-Pierfelice-Vallarino \cite{MR2902129} 
proved the improved (polynomial) dispersive and Strichartz estimates, which is strong enough to imply global results for $1<p\le p_{conf}(n)$, even though such results have not been stated explicitly. Based on Tataru's (exponential) dispersive estimates \cite{MR1804518},
the global results for  $1<p\le p_{conf}(n)$ were explicitly stated and proved by
 Sire-Sogge-Wang \cite{MR4026182}.
 An alternative proof of Tataru's dispersive estimates is available in
 Sire-Sogge-Wang \cite{MR4026182} for dimension three and the authors \cite{WZ22} for general spatial dimension $n\ge 2$. 
The nonshifted wave equations (with 
$\partial_{t}^2 u-\De_{\hn}$
instead of
$\partial_{t}^2 u-(\De_{\hn}+\rho^2)$)
 have also been investigated in
 Metcalfe-Taylor
\cite{MR2775816}, \cite{MR2944727}, and Anker-Pierfelice \cite{MR3254350}. 
 See also Anker-Pierfelice-Vallarino
\cite{MR3345662} for similar results on
 {D}amek-{R}icci spaces, as well as the recent works of Sire-Sogge-Wang-Zhang \cite{MR4169670}, \cite{MR4432952} for similar results on asymptotically hyperbolic manifolds. All these results show that the critical power is $p_c=1$, or we can say that there is no critical powers on hyperbolic spaces with power-type nonlinearities.


Thanks to the $\sinh r$ factor in the metric,
 we expect exponential decay of (linear) solutions, see, e.g., \cite{MR1443052}, \cite{MR1804518}, \cite{WZ22}  or
  Lemma \ref{formulae}. More precisely, by \eqref{radialsolu} in the appendix,
we are convinced that,
for smooth data with compact support,
 the linear solution behaves like $(\sinh t)^{-\rho}\sim e^{-\rho t}$ near the light cone $t=r$ as $t$ goes to infinity, at least when $n=3$. 
In light of the exponential decay of linear solutions, to illustrate the critical nature, 
it is natural to introduce the logarithmic nonlinearity, 
 which behaves like $\left(\ln {1}/{|u|}\right)^{1-p}|u|$ near $u=0$, for some $p>1$.
 One typical example
is \beeq\label{eq-nlterm}
F_p(u)=\left(\sinh^{-1} \frac{1}{|u|}\right)^{-(p-1)}|u|\ ,\eneq 
which behaves like
 $(\ln\frac{1}{|u|})^{1-p}|u|$ for small $|u|$ and $|u|^{p}$ for large $|u|$.

Concerning the problem of
global existence vs blow up for the
 Cauchy problem \eqref{nlw} with $F=F_p(u)$,
we expect there exist a critical power $p_c(n)>1$ and it is interesting to  determine the critical power $p_c(n)$ for any $n\ge 2$.


In this paper, we will concentrate on the physical case $n=3$. At first, concerning the problem of global existence with small data, we need only to assume the behavior of $F_p$ near $0$, that is, 
$F_p\in C^1$, $F_p(0)=F'_p(0)=0$, and
\beeqa\label{nonlinear}
|F_p'(u)|\lesssim \left(\ln\frac{1}{|u|}\right)^{1-p},\ \forall 0<|u|\ll 1\ .
\eneqa
Our first main result is the following.

\begin{theorem}[Global existence]\label{thm-sdge}
Let $n=3$ and $p>3$. Considering \eqref{nlw} with 
$F(u)=F_p(u)$ satisfying \eqref{nonlinear},
there exists $\vep_0(p)>0$ so that
the problem admits a global weak solution $u\in C(\R\times \hth)$ for any $|\vep|<\vep_0$
and initial data $(u_0, u_1)\in C^1(\hth) \times C(\hth)$
 satisfying \eqref{data}.
\end{theorem}

To determine the critical power $p_c(3)$, we consider the problem of blow up for relatively small powers. It turns out that $p_c(3)=3$ for
\eqref{nlw} with \eqref{eq-nlterm}, which is ensured by the following blow up result.

\begin{theorem}[Formation of singularity]\label{thm-bu}
Let $1<p<3$.
Considering \eqref{nlw} with compactly supported 
$C^1\times C$ data and
$F(u)=F_p(u)$ given by \eqref{eq-nlterm},
then the only global solution is the trivial solution.
In other words,
for any nontrivial data $(u_0,u_1)\in C_c^1\times C_c$
and arbitrary $\vep>0$, the corresponding weak solution will blow up in finite time.
\end{theorem}

Actually, similar to the global result, our proof could be adapted for general nonlinearities: we assume $F$ is a convex $C^1(\R)$ function so that,
$F(0)=F'(0)=0$,
\beeq\label{nonlinearity}
\left\{\begin{array}{l}
F(u)\gtrsim \left(\ln\frac{1}{|u|}\right)^{1-p}|u|,\ |u|\ll 1\ ,\\
F(u)\gtrsim |u|^{q},\ |u|\gtrsim 1\ ,
\end{array}
\right.
\eneq
 for some $q>1$ and $p\in (1,3)$.

At last, we would like to discuss some further problems, before concluding the introduction. 
Concerning the problem
\eqref{nlw} with \eqref{eq-nlterm},
the first natural problem is to determine the critical powers $p_c(n)$ for $n\neq 3$. For this problem, heuristically, in view of the sharp linear decay of the form $(\sinh t)^{-\rho}\sim e^{-\rho t}$ ($t>1$), we expect similar asymptotic behavior 
$u(t,x)\sim e^{-\rho t}$, along the light cone,
for $p>p_c(n)$, from which the nonlinear problem  \eqref{nlw}  is expected to behave like 
$$|\partial_{t}^2 u-(\De_{\hn}+\rho^2)u|=|F(u)|\les \<t\>^{-(p-1)} |u|\ .$$
Viewing the multiplication operator $\<t\>^{-(p-1)}$ as a short range perturbation of the operator $\pt^2$, it seems natural to
conjecture that $p_c(n)$ is precisely $3$, regardless of the spatial dimension. More precisely,
we conjecture that there exists
$\de=\de(n)>0$ so that 
we have global existence, with small data, for any $p\in (3, 3+\de(n))$, while for
$p\in (1,3)$, there exist
 some data $(u_0, u_1)$ so that there is no global solutions for any $\vep>0$.
 For the case with $1<p<p_c(n)$, besides the blow up results, it is also interesting to determining the sharp lifespan, in terms of $\vep$. Furthermore, the more challenging problem may be to understand the critical behavior when $p=p_c(n)$.

\subsection*{Organization of this paper} Our paper is organized as follows. We recall 
the fundamental dispersive estimate for the linear solution $u^0$ in Section 2. In Section 3, we prove the global existence by iteration, for any $p>3$, by exploiting the dispersive estimate.
The result for the formation of singularity, Theorem \ref{thm-bu}, is presented in Section 4, for which we closely follow the idea of John \cite{MR535704}. Finally, 
in the appendix, we present an elementary proof for the solution representation formula \eqref{mean}.

\subsection*{Notation} 
\begin{itemize}
\item
We use $A\lesssim B$ to denote
$A\leq CB$ for some large constant C which may vary from line to
line and depend on various parameters, and similarly we use $A\ll B$
to denote $A\leq C^{-1} B$. We employ $A\sim B$ when $A\lesssim
B\lesssim A$.
\item $d(x,y)$ is the geodesic distance between $x,y$ in $\hth$, and if $x$ is the origin $O$, we denote $|y|=d(O,y)$.
\item $S_t(x):=\{y\in\hth,d(x,y)=t\}$ denotes the hyperbolic sphere with center $x$ and radius $t$.
\item $M^r f(x)=\frac{1}{|S_r(x)|}\int_{S_r(x)}f(y)d\sig_y$ denotes the spherical mean of $f$ over $S_r(x)$. If the center is the origin $O$, for any function $w(t,x)$ with parameter $t$, we simply denote $\wt{w}(t,r)
:=M^r(w(t, \cdot))(O)$. 
\item For $(O,t_0)\in \hth\times\R$, we denote the forward and backward cones with vertex $(O,t_0)$ by
$$\Gamma^{\pm}(O,t_0)=\{(x,t): d(O,x)\le \pm (t-t_0)\}\ .$$
\end{itemize}

\section{Preliminary}\label{pre}

Inside the forward light cone of the Minkowski space $\La=\{(\tau,z)\in \R^{1,n}: |z|<\tau\}$, we introduce coordinates 
\beeq\label{varitrans}
s=|z|,\ \tau=e^t\cosh r,\ s=e^t\sinh r,\ r\in[0,\infty),\ t\in\R.
\eneq
Viewing $\hn$ as
the embedded spacelike hypersurface with $t=0$, we have
the natural metric $g_{\hn}=dr^2+(\sinh r)^2d\om^2$, induced from the Minkowski metric $g=-d\tau^2+dz^2=-d\tau^2+ds^2+s^2d\om^2$, where  $\om\in\ms^{n-1}$. This illustrates that $(r,\om)$ is the natural geodesic polar coordinates in $\hn$.

Considering the linear wave equations \beeq\label{wave}
\begin{cases}
\partial_{t}^2 u-(\De_{\hn}+\rho^2)u=F\ , x\in\hn\ ,\\
u(0,x)= u_0, u_t(0,x)= u_1\ ,
\end{cases}
\eneq
Duhamel's principle tells us that
 \eqref{wave} is equivalent to the integral equation
\beeq\label{ie}
u(t,x)= u^0(t,x)+(LF)(t,x)= u^0(t,x)+\int_0^t I(\tau, x, F(t-\tau))d\tau,
\eneq
where $u^0= \pt I(t,x, u_0)+  I(t,x,u_1)$, and $I(t,x,u_1)$ is the solution for the linear homogeneous equation with data $u_0=0$.

In the proof of Theorem \ref{thm-sdge},
a fundamental result to be used is the following a priori dispersive estimate for the linear solution $u^0$, which is available in G\"{u}nther \cite[Theorem 3.1]{MR1295101} or Fontaine \cite[Theorem 6]{MR1443052}.
For completeness, we present a proof.

\begin{lem}[Linear estimates]\label{formulae}
Let $n=3$ and $k>0$,
there exists $N_k>1$ so that
we have the estimate 
\beeq\label{eq-linear}
|u^0(t,x)|\le  N_k(\cosh |x|)^{-1}(\cosh (|t|-|x|))^{-k}\ ,
\eneq
for any solutions to  \eqref{wave} with $F=0$, whenever the  initial data $(u_0, u_1)\in C^1\times C$ satisfying \eqref{data}.
\end{lem}
\begin{proof}
Without loss of generality, we assume $t>0$. At first, we recall that 
$u^0(t,x)=\partial_tI(t, x, u_0)+I(t, x, u_1)$
on $(0,\infty)\times\hth $ with
\beeqa\label{mean}
I(t,x, u_1)=\sinh t\cdot (M^t u_1)(x),
\eneqa
which could be obtained 
from a relation between the wave operators on hyperbolic space and that on Minkowski space, see Appendix for a sketch of the proof.
 
As $|u_1(x)|\le \theta_k(x)=(\cosh |x|)^{-k-1}$, we have the following estimate with $|x|=r$
$$|I(t,x,u_1)|\le I(t, x,\theta_k)=\sinh t \cdot (M^t \theta_k)(x)\ ,$$
where
\beqa
(M^t \theta_k)(x)
&=&\frac{1}{2\sinh t \sinh r}\int_{|r-t|}^{r+t}(\cosh\la)^{-k-1}\sinh \la d\la\\
&=&\frac{(\cosh (t-r))^{-k}-(\cosh (t+r))^{-k}}{2k\sinh t \sinh r}\ .
\eeqa
We claim that we could prove an even better estimate for $|I(x,t,g)|$:
\beeq\label{eq-linear-claim}
|(M^t \theta_k)(x)|\le C_k \frac{1}{(\cosh t)(\cosh r) (\cosh (t-r))^{k}}\ .
\eneq

Before proving \eqref{eq-linear-claim}, let us check that it is strong enough to conclude \eqref{eq-linear}. Actually, when $u_0=0$, it is stronger than
\eqref{eq-linear}, due to the fact that $\tanh t\in [0, 1]$.
For the case with $u_1=0$, by \eqref{mean}, we see that
$$u^0=
\partial_t I(t, x, u_0)=(\cosh t) (M^t u_0)(x)+(\sinh t )\pt (M^t u_0)(x)\ ,
$$ for which it remains to control
$\pt (M^t u_0)(x)$.

To control $\pt (M^t u_0)(x)$, we introduce
a Lorentz boost $\psi_x\in SO(1,3)$ such that $\psi_x(O)=x$. It is known that $\psi_x S_t(O)=S_t(x)$ preserving the metric and,
for fixed $x, y$, $\psi_x (t,y)=(r_x(t,y),\om_x(t,y))=\gam_{x,y}(t)$ is a geodesic curve with
$|\pt\gam_{x,y}(t)|=1$. Then
\beqa
\partial_t(M^t u_0)(x) &=&\frac{1}{4\pi}\int_{S_1(O)} \partial_t(u_0(\psi_x(t,y))) d\sig_y\\
&=&\frac{1}{4\pi}\int_{S_1(O)} \partial_t(u_0(r_x(t,y),\om_x(t,y))) d\sig_y\\
&=&\frac{1}{4\pi}\int_{S_1(O)} \<\nabla u_0, \pt \gam_{x,y}(t)\>_{g_{\hth}} 
d\sig_y\ ,
\eeqa
and so
$$|\partial_t(M^t u_0)(x)|
\le\frac{1}{4\pi}\int_{S_1(O)} |\nabla u_0(\gam_{x,y}(t))|d\sig_y
\le M^t( \theta_k)(x)\ ,
$$
thanks to the assumption \eqref{data}.

To conclude the proof, we prove \eqref{eq-linear-claim}.
Actually, when $t\ge r$ and $r\ll1$, as
$\tanh \la$ is increasing and $\sinh r\ge r$, we have
\beqa
(M^t \theta_k)(x)
&=&\frac{1}{2\sinh t \sinh r}\int_{|r-t|}^{r+t}(\cosh\la)^{-k}\tanh \la d\la\\
&\le&\frac{\tanh (t+r)}{2 r \sinh t }\int_{t-r}^{t+r}(\cosh\la)^{-k} d\la\\
&\le&
\frac{\tanh (t+r)}{\sinh t }
 (\cosh (t-r))^{-k}\sim (\cosh t)^{-1}(\cosh r)^{-1}(\cosh (t-r))^{-k}.
\eeqa
Similarly, for $r>t$ and $t\ll1$, we obtain
\beqa
(M^t \theta_k)(x)
&=&\frac{1}{2\sinh t \sinh r}\int_{r-t}^{r+t}(\cosh\la)^{-k}\tanh \la d\la\\
&\le&\frac{\tanh (t+r)}{2 t \sinh r }\int_{r-t}^{t+r}(\cosh\la)^{-k} d\la\\
&\les&\frac{\tanh (t+r)}{ \sinh r }(\cosh (t-r))^{-k}\sim (\cosh t)^{-1}(\cosh r)^{-1}(\cosh (t-r))^{-k}\ .
\eeqa
Finally, for the remaining case with $r, t\gtrsim 1$, we have
Else, if $t\ge r\gtrsim 1$, it is clear that
\beqa
(M^t \theta_k)(x)&=&\frac{(\cosh (t-r))^{-k}-(\cosh (t+r))^{-k}}{2k\sinh t \sinh r}\\
&\le&\frac{(\cosh (t-r))^{-k}}{2k\sinh t \sinh r}
\sim (\cosh t)^{-1}(\cosh r)^{-1}(\cosh (t-r))^{-k}\ ,
\eeqa
which finish the proof of  \eqref{eq-linear-claim}.
\end{proof}


\section{Global existence} 

In this section, we give the proof of Theorem \ref{thm-sdge}, for which we rewrite \eqref{nlw} into the following integral equation
\beeq\label{ie2}
u(t,x)= \vep u^0(t,x)+(LF_p(u))(t,x)= \vep u^0(t,x)+\int_0^t I(\tau, x, F_p(u(t-\tau, \cdot))d\tau,
\eneq
where $u^0= \pt I(t,x, u_0)+  I(t,x,u_1)$ is the homogeneous solution with data $(u_0, u_1)$.

By Lemma \ref{formulae}, we have
\beqa
|u^0(t,x)|\le   N_k(\cosh |x|)^{-1}(\cosh (t-|x|))^{-k}\ ,
\eeqa
for any
$u_0\in C^1$, $u_1\in C$ satisfying \eqref{data}.

Let $\< t\>=\sqrt{1+t^2}$,
based on the elementary inequality
\beqa
\frac 12  e^{|t|}\le \cosh t\le e^{|t|},
\forall t\in\R\ ,
\eeqa
we observe that, for any $h>0$, there exists a constant $N_h>N_k$ such that
\beeqa\label{linear}
\overline{u^0}(t,r)\le  \frac{ N_k}{(\cosh r)(\cosh (t-r))^{k}}\le  
\frac{N_h}{e^{r}\<t-r\>^{h}}
:=\frac{  N_h}{\Phi_h(t,r)}\ ,
\eneqa
where we denote $\ol{f}(t,r)=\max_{|x|=r}|f(t,x)|$.

For fixed $h>0$ to be specified later,
the global existence of the solution $u$ of \eqref{ie2} will be proved by iteration, for which we have to introduce a suitable norm. We define the (complete) solution space and the 
solution map
 $T u$ as follows
$$X_{\vep} = \{u\in C([0,\infty)\times \hth): \|u\|:=\|\Phi_hu\|_{L^\infty_{t,x}}\le 2\vep N_h\}\ ,$$
$$(T u)(t,x)=\vep u^0+LF_p(u)\ .$$
Then 
 the proof is reduced to the following key nonlinear estimates,
in light of Banach's contraction principle.

\begin{lem}[Nonlinear estimates]\label{mainlem}
Let $p>3$, $h\in(1,p-2)$ and
$F_p$ be the $C^1$ function satisfying \eqref{nonlinear}. There exists $\vep_0>0$ so that
for any $\vep\in (0, \vep_0]$,
we have
\beeqa\label{compressed}
\|LF_p(u)-LF_p(v)\|\le \frac 12 \|u-v\|, \forall u, v\in X_\vep\ .
\eneqa
\end{lem}

Actually, with the help of Lemma \ref{mainlem}, we 
know that for $u\in X_{\vep_0}$,
$$\|LF_p(u)\|=\|LF_p(u)-LF_p(0)\|\le \vep_0 N_h\ ,$$
which tells us that
$$\|Tu\|\le \|\vep u^0\|+\|LF_p(u)\|\le 2\vep_0 N_h\ ,$$
i.e., $Tu\in X_{\vep_0}$.
In addition,
by \eqref{compressed},
we have
$$\|Tu-Tv\|=\|LF_p(u)-LF_p(v)\|\le \frac 12 \|u-v\|\ ,$$
which ensures that
$T: X_{\vep_0}\to X_{\vep_0}$ is a contraction map, and the fixed point is the desired solution.

\subsection{Proof of Lemma \ref{mainlem}}

By \eqref{nonlinear}, there exists $A>2$ so that
\beeqa\label{nonlinear2}
|F_p'(u)|\le A \left(\ln \frac{1}{|u|}\right)^{1-p}\triangleq G(u), \forall |u|\le \frac 1 A\ .
\eneqa

For any $u\in X_\vep$, as $\Phi_h\ge 1$, we know that
$$\|u\|_{L^\infty_{t,x}}\le
\|\Phi_h u\|_{L^\infty_{t,x}}=
 \|u\|\le 2 \vep N_h\le \frac 1 A\ ,$$
provided that $\vep\in (0, 1/(2N_h A)]=(0, \vep_1]$, for which we assume in what follows.
Then for any $u, v\in X_\vep$,
in view of
\eqref{nonlinear2} and the monotonicity of $G$,  we get
$$|F_p(u)-F_p(v)|\le G(\max(|u|,|v|)) |u-v|\le
G\left(\frac{2\vep N_h}{\Phi_h}\right)  \ol{u-v} 
\ .
$$
Recall $\Phi_h(t,r)=e^{r}\<t-r\>^{h}\ge e^r$, we see that
$$|F_p(u)-F_p(v)(t,x)|\le A \left(\ln\frac{1}{2\vep N_h}+r\right)^{1-p}  \ol{u-v} (t,r)
\ .
$$

By \eqref{ie2}, \eqref{mean} and \eqref{eq-sph-ave},  as well as the fact
$$|\ol{u}(t,r)|\le \frac{\|u\|}{\Phi_h (t,r)}
\le\frac{\|u\|}{\<t-r\>^h\cosh r }
\ ,$$
we have 
\beqa
&&|(LF_p(u)-LF_p(u))(t,x)|\\
&\le&L \left(A \left(\ln\frac{1}{2\vep N_h}+\cdot\right)^{1-p}  \ol{u-v} \right)(t,r)\\
&=&\frac{A}{2\sinh r}\int_0^t\int_{|t-s-r|}^{t-s+r}\left(\ln\frac{1}{2\vep N_h}+\la\right)^{1-p}  \ol{u-v}(s, \la) \sinh \la  d\la ds\\
&\le&\frac{ A \|u-v\|}{2\sinh r}\int_0^t\int_{|t-s-r|}^{t-s+r}\left(\ln\frac{1}{2\vep N_h}+\la\right)^{1-p}\frac{\tanh \la}{\<s- \la\>^h} d\la ds \\
&\le&\frac{ A\tanh (t+r)}{2\sinh r}\|u-v\| \int_{|t-r|}^{t+r}\int_{-\be}^{t-r}\left(\ln\frac{1}{2\vep N_h}+\frac{\be-\al}{2}\right)^{1-p}\<\al\>^{-h}d\al d\be
\ ,
\eeqa
where we have introduced new variables of integration $\al=s-\la$, $\be=s+\la$.

With the help of the above estimate, the proof of
 \eqref{compressed} is then reduced to the proof of
  the following claim:
\beeq\label{claim}
J(t,r):=\int_{|t-r|}^{t+r}\int_{-\be}^{t-r}\left(\ln\frac{1}{\vep }+\be-\al\right)^{1-p}\<\al\>^{-h}d\al d\be\ll \frac{\tanh r}{\tanh (t+r)}\<t-r\>^{-h} ,\eneq
provided that
$p>3$, $h\in(1,p-2)$ and
 $\vep$ is sufficiently small.

Concerning \eqref{claim}, we divide the proof into 
three separate cases:
$r\le \min (1,t)$, $1\le r\le t$ and $r>t$.

\subsubsection{Case 1: $r\le \min (1,t)$}
In this case, we have $r\sim\sinh r\sim \tanh r$. As $h>1$, it is obvious that $\<\al\>^{-h}$ is integrable and so
\beeqa
\label{inte1}&&\int_{t-r}^{t+r}\int_{-\be}^{(t-r)/2}\left(\ln\frac{1}{\vep }+\be-\al\right)^{1-p}\<\al\>^{-h}d\al d\be\\
\nonumber&\les&\int_{t-r}^{t+r}\left(\ln\frac{1}{\vep }+\be\right)^{1-p}d\be\\
\nonumber&\les&r\left(\ln\frac{1}{\vep }+t-r\right)^{1-p}\ll r\<t-r\>^{-h},
\eneqa
thanks to the assumption $h\in(1,p-1)$. On the other hand, for the  part with $\al\ge (t-r)/2$, we have
\beeqa
\label{inte2}&&\int_{t-r}^{t+r}\int_{(t-r)/2}^{t-r}\left(\ln\frac{1}{\vep }+\be-\al\right)^{1-p}\<\al\>^{-h}d\al d\be\\
\nonumber&\sim&\<t-r\>^{-h}\int_{t-r}^{t+r}\int_{(t-r)/2}^{t-r}\left(\ln\frac{1}{\vep }+\be-\al\right)^{1-p}d\al d\be\\
\nonumber&\les&\<t-r\>^{-h}\int_{t-r}^{t+r}\left(\ln\frac{1}{\vep }+\be-(t-r)\right)^{2-p}d\be\\
\nonumber&\les&\<t-r\>^{-h}r\left(\ln\frac{1}{\vep }\right)^{2-p}\ll r\<t-r\>^{-h}.
\eneqa
With the help of \eqref{inte1}, \eqref{inte2}, and the fact that $\tanh (t+r)\le1$, we have
\beeq\label{case1}
J(t,r)\ll r\<t-r\>^{-h}\sim \<t-r\>^{-h}\tanh r\le\<t-r\>^{-h}\frac{\tanh r}{\tanh (t+r)}.
\eneq
\subsubsection{Case 2: ${1\le r\le t}$}
In this case, we have $\tanh r\sim 1$. Similar to \eqref{inte1}, as $h\in(1,p-2)$, we obtain
\beeqa
\label{inte3}&&\int_{t-r}^{t+r}\int_{-\be}^{(t-r)/2}\left(\ln\frac{1}{\vep }+\be-\al\right)^{1-p}\<\al\>^{-h}d\al d\be\\
\nonumber&\les&\int_{t-r}^{t+r}\left(\ln\frac{1}{\vep }+\be\right)^{1-p}d\be\\
\nonumber&\les&\left(\ln\frac{1}{\vep }+t-r\right)^{2-p}\ll\<t-r\>^{-h}\ .
\eneqa
For the  part with $\al\ge (t-r)/2$, 
similar to \eqref{inte2},
we have
\beeqa
\label{inte4}&&\int_{t-r}^{t+r}\int_{(t-r)/2}^{t-r}\left(\ln\frac{1}{\vep }+\be-\al\right)^{1-p}\<\al\>^{-h}d\al d\be\\
\nonumber&\les&\<t-r\>^{-h}\int_{t-r}^{t+r}\left(\ln\frac{1}{\vep }+\be-(t-r)\right)^{2-p}d\be\\
\nonumber&\les&\<t-r\>^{-h}\left(\ln\frac{1}{\vep }\right)^{3-p}\ll \<t-r\>^{-h}\ ,
\eneqa
as we are assuming $p>3$.
 
As for case 1, by \eqref{inte3}, \eqref{inte4}, we get
\beeq\label{case2}
J(t,r)\ll \<t-r\>^{-h}\sim \<t-r\>^{-h}\tanh r\le \<t-r\>^{-h}\frac{\tanh r}{\tanh (t+r)}.
\eneq

\subsubsection{Case 3: ${r>t}$}
In this case, we have $$\tanh (t+r)\le\tanh (2r)\les \tanh r,\ \be-\al\ge\be\ ,$$
and so
\beeqa
\label{inte5}J(t,r)&=&\int_{r-t}^{t+r}\int_{-\be}^{t-r}\left(\ln\frac{1}{\vep }+\be-\al\right)^{1-p}\<\al\>^{-h}d\al d\be\\
\nonumber&\le&\int_{r-t}^{t+r}\left(\ln\frac{1}{\vep }+\be\right)^{1-p}\int_{-\be}^{t-r}\<\al\>^{-h}d\al d\be\\
\nonumber&\les&\left(\ln\frac{1}{\vep }+r-t\right)^{2-p}\<r-t\>^{1-h}\ll\<r-t\>^{-h}\frac{\tanh r}{\tanh(t+r)}.
\eneqa
This completes the proof of \eqref{claim} and so is Lemma \ref{mainlem}.

\section{Formation of singularity}
In this section, we present the proof of Theorem \ref{thm-bu}, when
$F(u)=F_p(u)$ satisfies \eqref{nonlinearity}.
Since we will show blow up for any nontrivial data, we could set $\vep=1$ without loss of any generality.
 As an initial step, we give the local existence and uniqueness, for compactly supported $C^1\times C^0$ data.
\subsection{Local existence and uniqueness}\label{sec-sub-bu}
We give a sketch of the proof for $t\in [0,T]$ with certain sufficiently small $T\in (0,1]$.

 
  Assume that $u_0,u_1$ have their support in a ball $|x|\le r_0$, by \eqref{linear}, we have 
\beeqa\label{lin}
\overline{u^0}(t,r)\le
Ne^{-r}\chi_{r\le t+r_0}\le N \chi_{r\le t+r_0}
,
\eneqa for some $N>0$,
where $\chi$ is the characteristic function. Based on \eqref{lin}, we introduce an alternative norm
\beeqa\label{bunorm}
\|u\|=
\|  \ol u(t,r)\|_{L^\infty([0,T]\times \hth)}
\eneqa
and the complete metric space $$X_T=\{u\in C([0,T]\times \hth): \|u\|\le 2  N, \mathrm{supp}\ u(t)\subset \{r\le t+r_0\} \}\
 .$$
 As $F\in C^1$ with $F(0)=0=F^{\prime}(0)$, there exists $M>0$ such that
 $|F(u)-F(v)|\le M|u-v|$ and so
  $\ol{F(u)-F(v)}\le M\ol{u-v}$,
 for any $u, v\in X_T$.
 Thus, for any such $u, v\in X_T$,
 it follows that
\beeqa
\nonumber |LF(u)(t,x)-LF(v)(t,x)|&\le &\frac{1}{2\sinh r}\int_0^t\int_{|t-s-r|}^{t-s+r} \ol{F(u)-F(v)}(s,\la)\sinh \la d\la ds\\
\nonumber&\le&\frac{ M\|u-v\|}{2\sinh r}\int_0^t\int_{|t-s-r|}^{t-s+r} d\la ds\\
\nonumber&=&\frac{M\|u-v\|}{4 \sinh r}\int_{|t-r|}^{t+r}\int_{-\be}^{t-r} d\al d\be\ .
\eneqa
An elementary calculation tells us that
\beqa
\int_{|t-r|}^{t+r}\int_{-\be}^{t-r}d\al d\be\le 4t\min (t,r)\le 4tr\ .
\eeqa
Recalling $\sinh r\ge r$, if $MT\le 1/2$,
we see that
$$ 
|LF(u)(t,x)-LF(v)(t,x)|\le
\frac{Mtr \|u-v\| }{ \sinh r} \le MT\|u-v\| \le \frac{1}{2}\|u-v\| \ ,$$
which ensures that 
$$\|LF(u)-LF(v)\|\le  \frac{1}{2}\|u-v\|, \forall u,v\in X_T\ .$$
Then, it is clear that the map
$$(T u)(t,x)= u^0+LF(u)\ $$
is a contraction map on $X_T$, which ensures local existence and uniqueness.

\subsection{Blow-up of the solution}

Because of the convexity of $F(u)$, by \eqref{ie2} with $\vep =1$, we have
\beeqa\label{fml}
\wt{u}(t,r)=\wt{u^0}(t,r)+L\wt{F(u)}(t,r)\ge\wt{u^0}(t,r)+LF(\wt{u})(t,r)\ ,
\eneqa
which gives us
\beeqa\label{lfu}
\wt{u}(t,r)\ge\wt{u^0}(t,r)+\frac{1}{2\sinh r}\int_0^t\int_{|t-s-r|}^{t-s+r}F(\wt{u})(s,\la)\sinh \la d\la ds.
\eneqa
Let $R_{r,t}$ denote the domain of the integration
\beeqa\label{domain}
R_{r,t}=\{(\la,s):t-r<s+\la<t+r,s-\la<t-r,s>0\}
\eneqa
in the $(\la, s)$ plane, see Figure 1. 

\begin{figure}[H]
\centering
\includegraphics[width=0.8\textwidth]{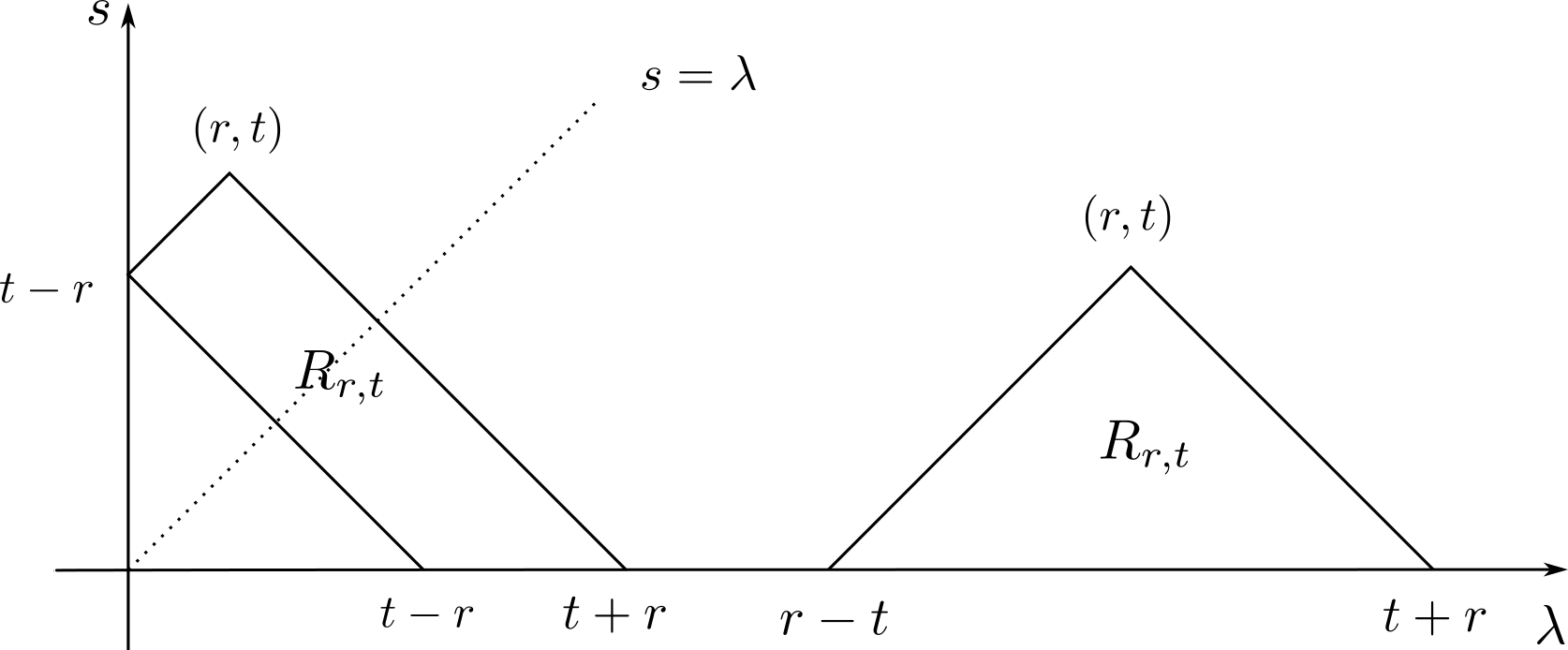}
\caption{Domain of the integration}\label{fig-1}
\end{figure}


The proof will be given by contradiction, for which we assume that there exists some nontrivial data $(u_0, u_1)\in C_c^1\times C_c$ so that the solution is global.

Let $t_0>0$ so that both the data are supported in  $\{x\in\hth: r\le t_0\}$. With the help of the sharp Huyghens' principle, we have $u^0=0$ in 
$\Gamma^{+}(O, t_0)$.
In the following, we shall prove that we must have \beeq\label{eq-claim-bu}\mathrm{supp}\ u\subset \Gamma^-(O,t_0)\ .\eneq
If this is true, then, by solving backward from $t=t_0+1$, the uniqueness result from
Subsection \ref{sec-sub-bu} tells us that $u \equiv 0$, which is clearly a contradiction to the nontrivial assumption on the data.

\subsection{Proof of  \eqref{eq-claim-bu} }
Suppose, by contradiction, that \eqref{eq-claim-bu} is not true, which means that there exists a point $(x_1,t_1)\notin \gm$ for which 
$$u(t_1,x_1)\neq0.$$
Set $t_2$ such that $t_2=t_1+|x_1|$, then $(O,t_2)\in \gp$. By \eqref{ie2} with $\vep =1$, the point $(x^1,t_1)$ lies inside the domain of integration of $L$ for $(t,x)=(t_2,O)$, which gives us
$$u(t_2,O)>\mu>0\ ,$$
for some $\mu>0$.
As $\wt{u}(t_2,0)=\wt{u}(t_2,O)>\mu$, by continuity, we can find a positive $\de>0$ so  that
\beeqa\label{eq-initial-lower}
\wt{u}(t, r)>\frac{\mu}2, \forall t\in [t_2, t_2+\de], r\le \de\ .
\eneqa

Before proceeding,
we introduce the following regions (see Figure 2)
\beeqa
\label{first}T&=&\{(\la,s):t_2+\de\le s+\la\le t_2+2\de,s-\la\le t_2\}\ ,\\
\label{second}S&=&\{(\la,s):t_2+2\de\le s+\la,t_2\le s-\la\le t_2+\de\}\ .
\eneqa

\begin{figure}[H]
\begin{tabular}{cc}
\begin{minipage}[t]{0.5\linewidth}
\centering
\includegraphics[width=0.8\textwidth]{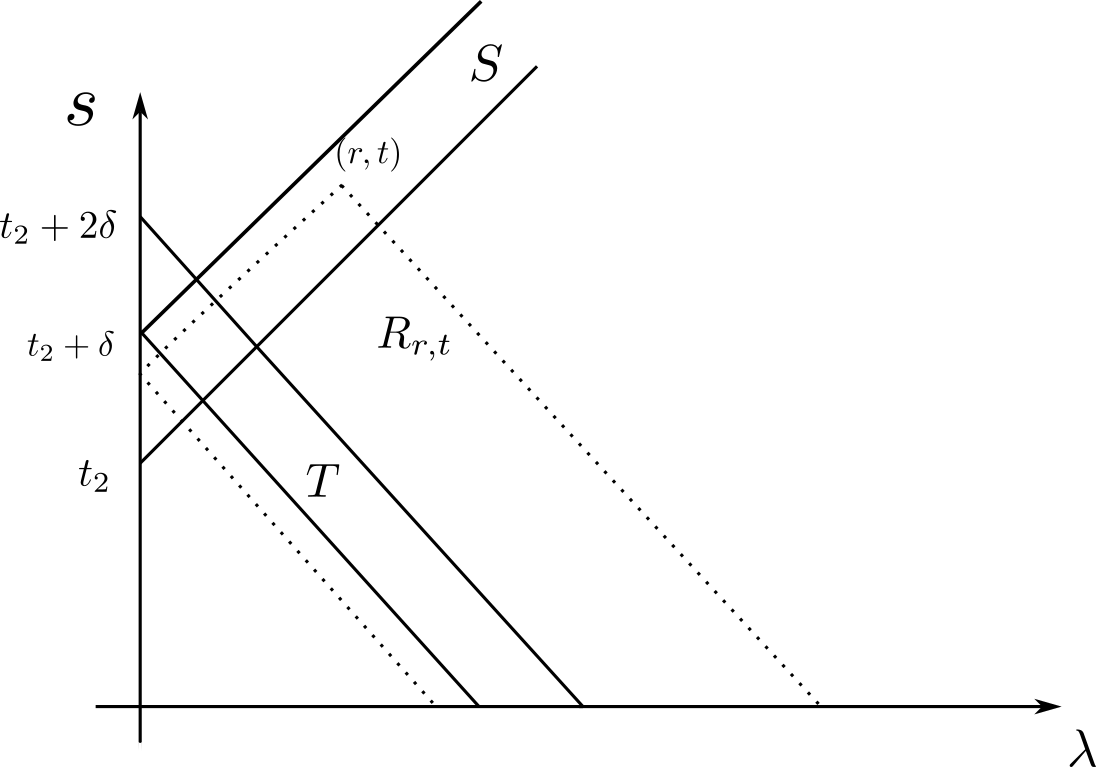}
\centerline{Figure 2}
\end{minipage}

\begin{minipage}[t]{0.5\linewidth}
\centering
\includegraphics[width=0.8\textwidth]{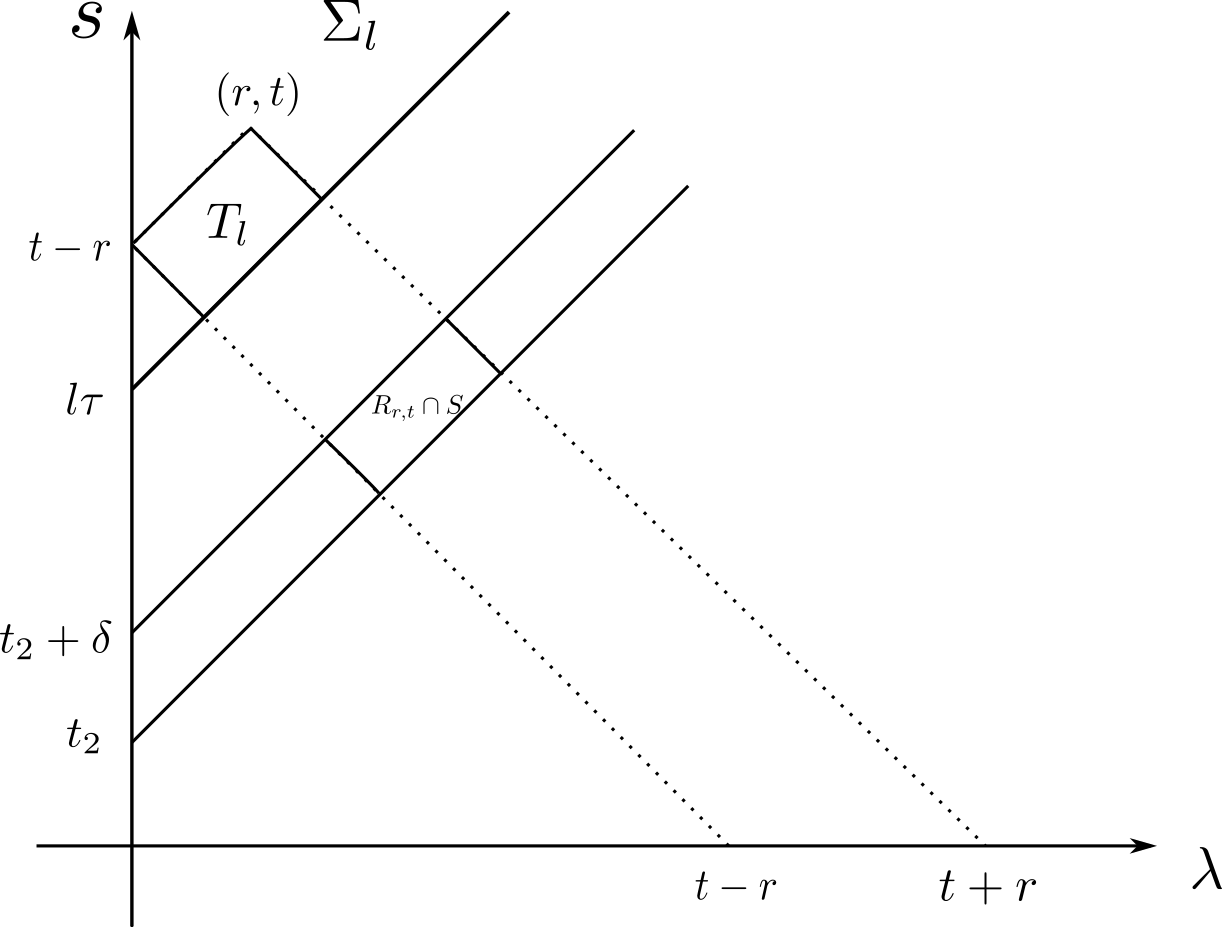}
\centerline{Figure 3}
\end{minipage}
\end{tabular}
\end{figure}

Observing that for any $(r,t)\in S$, we have
$T\subset R_{r,t}$.
Then
it follows from \eqref{lfu} and \eqref{eq-initial-lower} that for $(r,t)\in S$,
\beeqa\label{iter}
\wt{u}(t,r)\ge\frac{1}{2\sinh r}\int_{T}F(\wt{u})(s, \la)\sinh\la d\la ds\ge \frac{c_0}{\sinh r}\ ,
\eneqa
for some constant $0<c_0\ll1$, which is uniform in $(r,t)\in S$.

Recalling  \eqref{nonlinearity}, there exists $\de_0>0$ so that
\beeq\label{eq-nonlinearity2}
F(u)\ge \de_0\left(\ln\frac{1}{|u|}\right)^{1-p}|u|, \forall |u|<\de_0\ ;\ 
F(u)\ge \de_0|u|^{q}, \forall |u|>1/\de_0\  .\eneq
Without loss of generality, we could assume $c_0<\de_0\sinh (\de/2)$ so that, in view of \eqref{iter},
\beeq\label{eq-ite1}
F(\wt{u})(t,r)\ge\frac{ \de_0c_0}{\sinh r}\left(\ln\frac{\sinh r}{c_0}\right)^{1-p},\ 
\forall (r,t)\in S
\ .\eneq

\subsubsection{Improved lower bound}
To improve the lower bound, we introduce  the following regions for the $l$-th iteration, with $\tau=t_2+2\de$,
\beeqa
\label{third}\Si_l&=&\{(\la,s):s-\la>l\tau\}\ ,\\
\label{forth}T_l&=&\{(\la,s):t-r<s+\la<t+r,l\tau<s-\la<t-r\}\ .
\eneqa
See Figure 3 for an illustration.
Based on \eqref{eq-ite1} and \eqref{lfu}, for any $(r,t)\in\Si_1$, we could iterate once more to obtain
\beeqa
\nonumber\wt{u}(t,r)&\ge&\frac{1}{2\sinh r}\int_{R_{r,t}\cap S}F(\wt{u})(s,\la)\sinh \la d\la ds\\
\nonumber&\ge&\frac{\de_0 c_0}{2\sinh r}\int_{R_{r,t}\cap S}\left(\ln\left(\frac{\sinh\la}{c_0}\right)\right)^{1-p}d\la ds\\
\nonumber&\gtrsim&\frac{1}{\sinh r}\int_{t-r}^{t+r}\int_{t_2}^{t_2+\de}\left(\ln\left(\frac{1}{c_0}\right)+\be\right)^{1-p}d\al d\be\\
\nonumber&\gtrsim&\frac{r}{\sinh r}\left(\ln\left(\frac{1}{c_0}\right)+t+r\right)^{1-p},
\eneqa
which means that exists $c_1\in (0, c_0]$
\beqa
\wt{u}(t,r)\ge c_1\frac{r}{\sinh r}\left(\ln\left(\frac{1}{c_1}\right)+t+r\right)^{1-p}.
\eeqa

Suppose more generally that we have established an inequality of the form, 
\beeqa\label{form}
\wt{u}(t,r)\ge c_l\frac{r}{\sinh r}\left(t+r+\ln\left(\frac{1}{c_l}\right)\right)^{-b_l}(t-r)^{a_l},\  \forall\ (r,t)\in\Si_l,
\eneqa
for some   $b_l>0$, $c_l>0$ and
$a_l\in [ 0, b_l]$. Obviously, as $a_l-b_l\le0$, we could possibly take sufficiently small $c_l$ such that the lower bound is less than $\de_0$ and we could use the logarithmic term to iterate.
Based on \eqref{form},  a further iteration yields
\beqa
\nonumber\wt{u}(t,r)&\ge&\frac{1}{2\sinh r}\int_{T_l}F(\wt{u})(s,\la)\sinh \la d\la ds\\
\nonumber&\gtrsim&\frac{1}{\sinh r}\int_{t-r}^{t+r}\int_{l\tau}^{t-r}\left(\ln\left(\frac{1}{c_l}\right)+\be\right)^{1-p}\left(\be+\ln\left(\frac{1}{c_l}\right)\right)^{-b_l}\al^{a_l}(\be-\al)d\al d\be\\
\nonumber&\ge&\frac{1}{\sinh r}\int_{t-r}^{t+r}\left(\ln\left(\frac{1}{c_l}\right)+\be\right)^{-b_l+1-p}\int_{l\tau}^{t-r}(\al-l\tau)^{a_l}(t-r-\al)d\al d\be\\
\nonumber&\gtrsim&\frac{1}{\sinh r}(t-r-l\tau)^{a_l+2}\int_{t-r}^{t+r}\left(\ln\left(\frac{1}{c_l}\right)+\be\right)^{-b_l+1-p}d\be\\
\label{liter}&\gtrsim&\frac{r}{\sinh r}(t-r-l\tau)^{a_l+2}\left(\ln\left(\frac{1}{c_l}\right)+t+r\right)^{-b_l+1-p}\ ,
\eeqa
for any $(r,t)\in \Si_{l}$. If we assume $(r,t)\in \Si_{l+1}$, we get
$t-r-l\tau\sim t-r$ and so is 
 \eqref{form} with $a_{l+1}=a_l+2$, $b_{l+1}=p-1+b_l$ and some $c_{l+1}\in (0, c_l)$.
 
 By induction, 
 with
$a_1=0$ and $b_1=p-1$, it is clear that 
\eqref{form} with $l=j$ could be boosted to 
\eqref{form} with $l=j+1$, as long as $a_j\le b_j$.
As $a_j=2j-2$, $b_j=(p-1)j$, the procedure breaks in finite steps, if $1<p<3$. To be more specific, with $l_0:=\left[\frac{2}{3-p}\right]+1$, we have, for some $c>0$,
\beeqa\label{eq-postibd}
\wt{u}(t,r)\ge c \frac{r}{\sinh r}\left(t+r+\ln\left(\frac{1}{c}\right)\right)^{-l_0(p-1)}(t-r)^{2l_0-2},\ 
\eneqa
for all $(r,t)\in\Si_{l_0}$.
Here, $-l_0(p-1)+2l_0-2=l_0(3-p)-2>0$.

\subsubsection{Further improved lower bound}
Equipped with the lower bound \eqref{eq-postibd}, which blows up at infinity, we could exploit the power type nonlinearity to show blow up in finite time.

Let
$A_0=l_0(3-p)-2>0$,  and
$$
Y_0=\{(\la,s): \la\le1,s\ge T\}\ ,\ 
Y=\{(\la,s): \la\le1,  s-\la\ge T, s+\la\le T+2\}\ ,
$$
where $T>\left(\frac{1}{c}\right)^{1/A_0}$ is a constant to be determined later so that
$Y_0\subset\Si_{l_0}$ (see Figure 4).

Restricted to $Y_0$,
the lower bound
 \eqref{eq-postibd} tells us that
\beeqa\label{fstep}
\wt{u}(t,r)\ge \tilde c\ t^{A_0}\ge \tilde c\ T^{A_0},
\eneqa
for some $\tilde c>0$. We shall require $\tilde c\ T^{A_0}>\de_0^{-1}$ so that we could
apply the power type nonlinearity \eqref{eq-nonlinearity2}:
$$F(\wt{u})\ge \de_0 |\wt{u}|^q,\ \forall (r,t)\in Y_0\ .$$

\begin{figure}[H]
\centering
\includegraphics[width=0.4\textwidth]{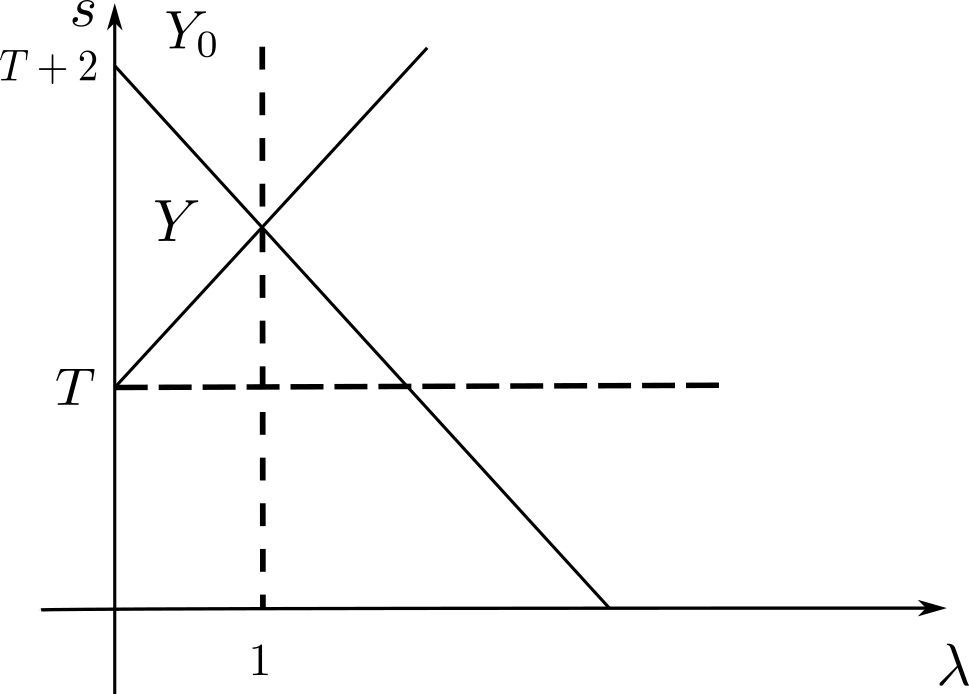}
\centerline{Figure 4}
\end{figure}
As before, we would like to boost \eqref{fstep} to illustrate blow up in finite time.
For such purpose, suppose that we have a lower bound of the following form
\beeqa\label{newform}
\wt{u}(t,r)\ge D T^{A} (t-r-T)^{B}\ ,
\eneqa
 for $(r,t)\in Y\subset Y_0$.
Then
\begin{eqnarray*}
\wt{u}(t,r)&\ge&\frac{1}{2\sinh r} \int_{R_{r,t}\cap Y}F(\wt{u})(s,\la)\sinh \la d\la ds\\
&\ge & \frac{\de_0}{2\sinh r}
\int_{R_{r,t}\cap Y}|\wt{u}(s,\la)|^q\sinh \la d\la ds\\
&\ge &\frac{\de_0 }{8\sinh r}
\int_{t-r}^{t+r}\int_{T}^{t-r} D^qT^{Aq} (\al-T)^{B q}(\be-\al)d\al d\be
\\
 &=&
 \frac{\de_0 D^qT^{Aq}}{8\sinh r}
\int_{t-r}^{t+r}\int_{T}^{t-r}(\al-T)^{B q}((\be-T)-(\al-T))d\al d\be
\\
 &\ge& \frac{\de_0  D^q r}{4
 (B q+1)(B q+2)\sinh r}T^{Aq}(t-r-T)^{Bq+2}\\& \ge&
 \frac{\de_0  D^q }{8
 (B q+2)^2}T^{Aq}(t-r-T)^{Bq+2}\ ,
\end{eqnarray*}
as $\sinh r\in [r, 2r]$ for $r\in [0,1]$.

As we know \eqref{newform} with $D=\tilde c$, $A=A_0$, $B=0$,
by induction, we have
\eqref{newform} with $D=D_m$, $A=A_m$ and $B=B_m$
for any $m\ge 0$, provided that
$D_0=\tilde c$, $B_0=0$,
and
\beeqa\label{m+1}
A_{m+1}=A_m q, B_{m+1}=B_m q+2, D_{m+1}
=
 \frac{\de_0  D_m^q}{8 (B_mq+2)^2}=
  \frac{\de_0  D_m^q}{8 B_{m+1}^2}
 \ .
\eneqa
  Solving \eqref{m+1} yields for $m\ge1$
\beeqa\label{law1}
A_m=A_0 q^m,\  B_m=2\frac{q^m-1}{q-1}\le 2mq^{m-1},\  D_m\ge
\frac{\de_0D_{m-1}^q}{32 m^2q^{2(m-1)}},
\eneqa
and thus
\beqa
D_m\ge \exp\left[q^m\left(\ln D_0-\sum_{j=0}^{m-1}\frac{2\ln(j+1)+2j\ln q-\ln\frac{\de_0}{32}}{q^{j+1}}\right)\right].
\eeqa
Let
\beeqa\label{const}
E=\ln D_0-\sum_{j=0}^{\infty}\frac{2\ln(j+1)+2j\ln q-\ln\frac{\de_0}{32}}{q^{j+1}}\ ,
\eneqa
for which the convergence is ensured by $q>1$,
it follows that, for any $m\ge 1$,
\beeqa\label{law2}
D_m\ge \exp(Eq^m)\ .
\eneqa

Then by \eqref{law1},\eqref{law2},\eqref{const}, we have for $(r,t)\in Y$ and sufficiently large $m$
\beqa
\wt{u}(t,r)\ge \exp\left[q^m\left(E+A_0\ln T+\frac{2}{q-1}\ln(t-r-T)\right)\right](t-r-T)^{-\frac{2}{q-1}}.
\eeqa
Let $r=0$ and $t=T+2$, 
the term
$E+A_0\ln T+\frac{2}{q-1}\ln(t-r-T)
$ is positive, for sufficiently large $T$. Then, for such $T$, it follows that $u(T+2,O)
=\wt{u}(T+2,r)\to\infty$ as $m\to\infty$, which is the desired contradiction.

\section{Appendix}
In this section, we would like to present an elementary proof for the formula \eqref{mean}. By the coordinates \eqref{varitrans} in Section \ref{pre}, a simple computation leads to the following connection between wave operators on  hyperbolic space and that on Minkowski space:
$$\square=\partial_{\tau}^2-\De_{\R^3}=e^{-3t}(\partial_t^2-(\De_{\hth}+1))e^t=e^{-3t}\square_{\hth}e^t\ .$$
Without loss of generality, we need only to prove
the formula \eqref{mean} for $x=O$, for which we use the geodesic polar coordinates $(r,\omega)$.

Let
$u=I(t,x,u_1)$, we know that it satisfies
$$(\partial_t^2-\partial_r^2-\frac{2}{\tanh r}\partial_r-
\frac{1}{(\sinh r)^2}\Delta_\omega 
-1)u=0, u(0,x)=0, u_t(0,x)=u_1\ .$$
Taking spherical average on $S_1(O)$,  we see that
$U(t,r)=M^r(u(t,\cdot))(O)$ satisfies
$$(\partial_t^2-\partial_r^2-\frac{2}{\tanh r}\partial_r
-1)U=0, U(0,r)=0, \pt U (0,r)=(M^r u_1)(O):=G(r)\ .$$

Let $W(\tau,s)=e^{-t} U(t,r)$ with $\tau^2-s^2=e^{2t}$, $s/\tau=\tanh r$,  we have
\beqa
\begin{cases}
\square W=e^{-3t}(\partial_t^2-\partial_r^2-\frac{2}{\tanh r}\partial_r-1)U=0\\
W|_{t=0}=W(\<s\>,s)=0,\ W_t|_{t=0}
=G(\tanh^{-1}\frac{s}{\<s\>})
\end{cases},
\eeqa
where $\partial_t=\<s\>\partial_{\tau}+s\partial_s$, $\partial_r=s\partial_{\tau}+\<s\>\partial_s$ when $t=0$. On the hyperbolic space $t=0$, we see that  $\partial_{\tau}=\<s\>\partial_t-s\partial_r$, $\partial_s=\<s\>\partial_r-s\partial_t$ and thus 
\beeq\label{eueq}
\left\{\begin{array}{l}
\square W=(\partial_{\tau}^2-\partial_s^2)(sW)=0\\
W(\<s\>,s)=0,\ W_{\tau}(\<s\>,s)=(\<s\>W_t-sW_r)|_{t=0}=\<s\>G(\tanh^{-1}\frac{s}{\<s\>})\ .
\end{array}\right.
\eneq

Let $\al=\tau-s, \be=\tau+s$, $Z(\al,\be)=sW$,  and
$\la=\ln (\<s\>+s)$,  we can transform \eqref{eueq} further to the following form
\beeq\label{euequ}
\left\{\begin{array}{l}
\partial_{\al}\partial_{\be}Z=0\\
Z(e^{-\la},e^{\la})=0,\ Z_{\be}(e^{-\la},e^{\la})=\hf(\partial_{\tau}+\partial_s)(sW)=\frac{\sinh \la}{2e^{\la}}G(\la)\ .
\end{array}\right.
\eneq

In view of the d'Alembert's formula, as well as the fact that $Z|_{s=0}=0$, we obtain for $\al=\tau-s=e^{-\mu}$, $\be=\tau+s=e^{\nu}$ ($\mu\in\R,\ \nu\ge0$)
$$Z(e^{-\mu},e^{\nu})=\int_{e^{|\mu|}}^{e^{\nu}}\frac{\sinh \la}{2e^{\la}}G(\la)d(e^{\la})=\hf\int_{|\mu|}^{\nu}G(\la)\sinh \la d\la\ .$$
Finally, as $e^{2t}=\tau^2-s^2=e^{\nu-\mu}$, $e^{2r}=\frac{\tau+s}{\tau-s}=e^{\nu+\mu}$, we have
\beeq\label{radialsolu}
U(t,r)=\frac{e^t}{s}Z(e^{-\mu},e^{\nu})=\frac{1}{2\sinh r}\int_{|t-r|}^{t+r}G(\la)\sinh\la d\la.
\eneq
For $t>0$, we have $u(t,0)=
U(t,0)=\lim_{r\to0}U(t,r)=(\sinh t) G(t)=(\sinh t)(M^t u_1)(O)$, this completes the proof of  \eqref{mean}.

By the way, we remark that the formula \eqref{radialsolu} and \eqref{mean} tells us that, for radial functions $u_1$, we have the following formula for the spherical average:
\beeq\label{eq-sph-ave}(M^t u_1)(r)
=\frac{1}{2 \sinh t \sinh r}\int_{|t-r|}^{t+r}u_1(\la)\sinh\la d\la\ .\eneq
 
 \subsection*{Acknowledgment}
The first author would like to thank Professor Vladimir Georgiev for proposing the problem of logarithmic nonlinearity,
during the ``Waseda Workshop on Partial Differential Equations 2019" in Waseda University, where the first author gave a talk on the wave equations on hyperbolic spaces with power type nonlinearity.

\end{document}